\documentclass[12pt]{article}

\usepackage{amssymb}
\usepackage[leqno]{amsmath}
\usepackage{amsthm}
\usepackage{hyperref}
\usepackage{graphicx}
\usepackage{enumerate}
\usepackage{pictexwd,dcpic}

\newtheorem{theorem}{Theorem}[section]
\newtheorem{corollary}[theorem]{Corollary}

\newtheorem{lemma}[theorem]{Lemma}

\theoremstyle{definition}

\newtheorem{remark}[theorem]{Remark}

\newtheorem{definition}[theorem]{Definition}

\newcommand\coker{\operatorname{coker}}
\newcommand\image{\operatorname{image}}

\def\setR{\mathbb R}

\def\setC{\mathbb C}

\def\setP{\mathbb P}

\def\setZ{\mathbb Z}

\def\J{\mathcal J}
\def\M{\mathcal M}

\def\A{\mathcal{A}}
\def\J{\mathcal{J}} 

\newcommand{\delbar}{{\ensuremath{\bar\partial }}}
\def\id{\mathrm{Id}}
\newcommand{\norm}[1]{{\ensuremath{|\!|#1|\!|}}}
\def\dvol{d\mathrm{vol}}

\title{Existence and Stability of Foliations\\ by
  $J$--Holomorphic Spheres}
\author{R. Hind \and J. von Bergmann}
\date{\today}

\begin{document}

\maketitle
\abstract{We study the existence and stability of holomorphic
  foliations in dimension greater than 4 under perturbations of the
  underlying almost complex structure. An example is given to show
  that, unlike in dimension 4, $J$--holomorphic foliations are not
  stable under large perturbations of almost complex structure.  }

\section{Introduction}

The theory of pseudoholomorphic curves was introduced in Gromov's
seminal paper \cite{gromov}. There is a Fredholm theory showing that,
for generic almost-complex structures $J$, pseudoholomorphic, or
$J$-holomorphic, curves appear in finite dimensional families, with
the dimension given by the Riemann-Roch theorem. Furthermore, in the
presence of a taming symplectic form, suitable moduli spaces of
$J$-holomorphic curves are compact modulo bubbling. These results have
many important applications in symplectic topology. Notably they lead
to Gromov-Witten invariants and Floer homology, which have been the main
methods for establishing rigidity results in symplectic and contact
topology.

Applied to symplectic manifolds of dimension $4$ the theory of
pseudoholomorphic curves is especially powerful and it becomes
possible to prove classification results which are as yet inaccessible
in higher dimensions. For example, symplectic forms on $S^2 \times
S^2$ are classified, see \cite{SW_GW} and \cite{gromov}, their
symplectomorphism groups are well understood, see \cite{gromov} and
\cite{abreu_mcduff}, and the Lagrangian spheres are all known to be
symplectically equivalent, see \cite{hind_spheres}. These results all
rely on the existence of foliations by $J$-holomorphic spheres. More
precisely, they utilize the following theorem of Gromov. We say that a
homology class $A \in H_2(X)$ is $\omega$--minimal if $\omega(A) =
\min_{B \in H_2(X), \omega(B)>0} \omega(B)$.

\begin{theorem}\label{thm:4d-foliation}
  Let $(X,\omega)$ be a symplectic $4$-manifold with a tamed
  almost-complex structure $J$ and suppose that there exists an
  embedded symplectic sphere in a homology class $A$ satisfying $A
  \bullet A =0$ and $A$ is $\omega$--minimal.

  Then $X$ is foliated by the images of $J$-holomorphic spheres
  homologous to $A$. The foliations vary smoothly with the
  almost-complex structure $J$.
\end{theorem}

The aim of paper is to investigate the extent to which this remains
true when $X$ has higher dimension.

It turns out that in general Theorem \ref{thm:4d-foliation} is false
if $X$ is allowed to have dimension greater than $4$. The existence of
$J$-holomorphic spheres in the class $A$ can be guaranteed at least
for an open set of almost-complex structures by imposing an index
constraint. But even if a foliation is known to exist for a particular
$J$, it is unstable in the sense that varying $J$, even in the most
generic fashion, can cause the foliation to degenerate.

To be more precise, we recall that given a family of tame
almost-complex structures $J_t$, $0 \le t \le 1$, on the symplectic
manifold $X^{2n}$ we can define the universal moduli space
\begin{eqnarray}
{\cal M} = \{(u,t)|u:S^2 \to X\ J_t\mathrm{-holomorphic},\,[u(S^2)]=A
\}.\label{eq:M}
\end{eqnarray}
Suppose that
$c_1(A)=2$, then ${\cal M}$ has virtual dimension
$2n+5$ and if the family $\{J_t\}$ is regular then ${\cal M}$ is a manifold (with boundary) of dimension $2n+5$. Furthermore, in the generic case the projection map $T:{\cal M} \to [0,1]$, $(u,t) \mapsto t$ is a Morse function and for all but finitely many $t$ the fiber ${\cal M}_t$ is a manifold of dimension $2n+4$ consisting of the $J_t$-holomorphic spheres in the class $A$. For such regular $t$ there is a smooth evaluation map $$e_t:{\cal M}_t \times _G S^2 \to X$$ where the equivalence relation $G$ is
reparameterization of the holomorphic spheres. Both ${\cal M}_t \times _G S^2$ and $X$ are smooth $2n$-dimensional manifolds.

\begin{definition} We say that $X$ is {\it foliated} by
  $J_t$-holomorphic spheres in the class $A$ if the map $e_t$,
  when restricted to some connected component of its domain, is a
  homeomorphism.

  We say that $X$ is {\it smoothly foliated} by $J_t$-holomorphic
  spheres in the class $A$ if the map $e_t$, when restricted to some
  connected component of its domain, is a diffeomorphism.
\end{definition}

When $X$ is $4$-dimensional, or when the almost-complex structure
$J_t$ is integrable, these two notions coincide, but in higher
dimensions there exist foliations (at least if we allow nonregular
curves) for which the corresponding evaluation map is a
smooth homeomorphism that is not a diffeomorphism. An example is given
in Remark \ref{rmk:foln}.

The following result shows that Theorem \ref{thm:4d-foliation} fails completely in dimension greater than four. Let $(M,\omega_M)$ be a symplectic manifold of dimension at least four.

\begin{theorem}\label{thm:counterexample}
  There exists a regular family $J_t$ of tame almost-complex
  structures on $(X,\omega)=(S^2\times
  M,\sigma_0\oplus \omega_M)$ such that ${\cal M}$ has a
  component ${\cal N}$ where the curves in $t^{-1}(0) \cap {\cal N}$
  form a foliation of $X$ but the curves in $t^{-1}(1) \cap {\cal N}$
  do not, the curves are not disjoint.
\end{theorem}

In fact, we can take $J_0$ to be a product structure on $S^2\times M$
and so ${\cal M}_0$ has a single component consisting of curves with
images $S^2 \times \{z\}$ for $z \in M$. Fixing a point $0 \in M$ we
can further assume that the corresponding sphere $C_0 = S^2 \times
\{0\}$ is $J_t$-holomorphic and regular for all $t$. However there
exists a two parameter family of curves $C_r$ in ${\cal M}_1$ which
includes $C_0$ but with $C_r \cap C_0 \neq \emptyset$ for all $r$.

An analog of Theorem \ref{thm:4d-foliation} does remain true
if we impose restrictions on the $J_t$. In this paper we will explain
how to guarantee the existence and stability of foliations in the case
of integrable complex structures and additional restrictions on the
curvature.

\begin{theorem}\label{thm:integrable}
  Let $(X,\omega,J)$ be K\"ahler with holomorphic bisectional
  curvature bounded from below by $c>-\pi/\omega(A)$, where $A\in
  H_2(X;\setZ)$ is an $\omega$--minimal homology class with
  $GW^X_{0,1,A}(pt)=1$. Then $X$ is smoothly foliated by
  $J$--holomorphic spheres.
\end{theorem}

We remark that if $X$ is a product $(M,k\omega) \times (S^2, \sigma)$,
where $(M,\omega,J)$ is K\"ahler and $\sigma$ is the area form on
$S^2$, then a product complex structure will satisfy the hypotheses of
Theorem \ref{thm:integrable} whenever $k$ is sufficiently large, and
so will any other integrable complex structure that is sufficiently
close to the product one.

Of central importance to the stability of foliations is the notion of
superregularity as defined in \cite{donaldson_discs}.
\begin{definition}\label{def:superregular}
  A real--linear Cauchy--Riemann operator $D$ on a complex vector
  bundle over $S^2$
  is called {\em regular} if $D$ is surjective. It is called {\em
    superregular} if $\ker D$ contains a collection of sections that
  are linearly independent over each point in $S^2$. A choice of such a
  collection of sections is called a {\em superregular basis} for $D$.

  A $J$--holomorphic sphere $u$ is called regular if the induced
  real--linear Cauchy--Riemann operator $D_u$ on $u^\ast TX$ is
  regular. An immersed $J$--holomorphic sphere $u$ is called
  superregular if $D_u$
  acting on sections of the normal bundle is superregular.
\end{definition}
Note that regularity does not imply superregularity and vice
versa. For example, no regular linearized operator at a
$J$--holomorphic curve in a 4--manifold with self-intersection number
$\ne 0$ is superregular. We will give an example of a superregular
operator that is not regular in Section \ref{sec:constr-line-oper}.

Another way to understand what it means for a linearized operator to
be regular and superregular is the following.  Suppose $u$ is regular
so that the moduli space of curves near $u$ is a smooth
manifold. Then, in the language of Section 3.4 in \cite{mcduff2}, the
evaluation map from the moduli space of $J$--holomorphic curves near a
map $u:S^2\rightarrow X$ is transverse to all $x\in \image(u)\subset
X$ if and only if $u$ superregular. Hence all curves $u$ in a smooth foliation are superregular.

The paper is arranged as follows. We first establish the non-existence
result Theorem \ref{thm:counterexample} in Section
\ref{sec:counterexample}. Then we discuss the integrable case in
Section \ref{sec:stab-foli-integr} to prove Theorem
\ref{thm:integrable}.

\section{Non-stability of foliations}
\label{sec:counterexample}
For clarity of exposition we will restrict ourselves to work in
dimension 6. It is clear how to generalize this to higher dimension,
e.g. by taking the product with another symplectic manifold with
compatible almost complex structure. However, our construction does
not work in dimension less than 6 since in that case Hirsch's theorem
about immersions does not apply, and consequently Lemma
\ref{lem:D_super} does not hold.

\subsection{Superregular Operator with Cokernel}
\label{sec:constr-line-oper}

Here we will construct a superregular Cauchy-Riemann operator with nontrivial cokernel. This immediately gives examples of foliations by holomorphic spheres which are not superregular.

Throughout this section $N=S^2\times\setR^4$ denotes the trivial
bundle. Let $\{\bar e_i\}_{i=1}^4$ be the canonical basis of
$\setR^4$ and $J_0$ the canonical complex structure. Using the
trivialization of $N$ we will frequently identify sections of $N$ with
functions from $S^2$ into $\setR^4$.

Recall the structure of a real--linear Cauchy--Riemann $D$ operator
acting on sections $\xi$ of the complex vector bundle $(N,J_0)$ with
trivial connection $\nabla$ via
\begin{eqnarray*}
  D\xi=\frac12\left(\nabla \xi+J_0\nabla\xi\circ j\right)+\frac12 Y\xi
  =\delbar_{J_0}\xi+\frac12Y\xi
\end{eqnarray*}
where $Y:N\rightarrow\Lambda^{0,1}(T^\ast S^2\otimes_\setC N)$ is a
vector bundle homomorphism.

Recall Definition \ref{def:superregular} of our use of the terms
regular and superregular.
\begin{lemma}\label{lem:J(z)-hol}
  Let $D$ be a superregular real--linear Cauchy--Riemann operator on
  $(N,J_0)$ with superregular basis $\{e_i\}_{i=1}^4$. Let $\Phi:S^2\times
  \setR^4\rightarrow N$ be the corresponding trivialization,
  i.e. $\Phi(z,x)=\sum_{i=1}^n x_i\,e_i(z)$.

  Then for a function $f:S^2\rightarrow \setR^4$ we have $D
  \Phi(z, f(z))=\Phi_\ast\delbar_J f$, where
  $J:S^2\rightarrow End(\setR^4)$, is given by $\Phi^\ast J_0$.
\end{lemma}

\begin{proof}
  \begin{eqnarray*}
    D\Phi(z,f(z))=D\sum_{i=1}^4 f_i(z)e_i(z)
    =\sum_{i=1}^4\frac12\{df_i+J_0\,df_i\circ j\}e_{i}
    =\Phi_\ast(\delbar_J f).
  \end{eqnarray*}
\end{proof}

We need the following elementary observation.
\begin{lemma}\label{lem:Y}
  Given any four sections $\{e_i\}_{i=1}^4$ of $N=S^2\times \setC^2$
  that are linearly independent over each $z\in S^2$, there exists a
  unique real--linear Cauchy--Riemann operator $D=\delbar_0+Y$, where
  $\delbar_0$ is the canonical complex Cauchy--Riemann operator, $Y\in
  \Lambda^{0,1}T^\ast S^2\otimes \setC^2$, and $\{e_i\}_{i=1}^4$ is a
  superregular basis for $D$, i.e. so that $D e_i=0$ for $i=1,\ldots
  ,4$.
\end{lemma}
\begin{proof}
  Let $\nu_i=D_0\,e_i$. Since the $\{e_i\}_{i=1}^4$ are linearly
  independent for all $z\in S^2$ we may define $Y$ via
  $Y_z\,e_i(z)=-\nu_i(z)$. Thus $D e_i=0$ so $\{e_i\}_{i=1}^4$ is a
  superregular basis. Conversely, if $D e_i=0$ for $i=1,\ldots,4$ then
  $Y_z(e_i(z))=-\nu_i(z)$, defining $Y$ uniquely. 
\end{proof}

We now aim to construct a superregular real--linear Cauchy--Riemann
operator on $N$ that has a non--trivial cokernel. The following Lemma
clears some topological obstructions.
\begin{lemma}\label{lem:F_0}
  There exists a complex bundle monomorphism $F_0:TS^2\rightarrow
  \underline{\setC^2}$.
\end{lemma}
Here $\underline{\setC^2}$ denotes the trivial $\setC^2$--bundle over
$S^2=\setC\setP^1$.

\begin{proof}
  Consider the diagram in Figure \ref{fig:1}.
  Let $K\rightarrow \setC\setP^1$ be the tautological bundle,
  i.e.
  \begin{eqnarray*}
    K=\{(v,[z:w]|\,v\in \mathrm{span}_\setC(z,w),\quad
    (z,w)\in\setC^2\}
  \end{eqnarray*}
  and let $\tilde f_0:S^2\rightarrow \setC\setP^1$ be a degree 2
  map. Then $TS^2$ and $\tilde f_0^\ast K$ have the same Chern class,
  so they are isomorphic complex vector bundles. Let $\tilde
  F_0:TS^2\rightarrow K$ be a bundle isomorphism (covering $\tilde
  f_0$).

  Let $\iota:K\rightarrow \underline{\setC^2}$ be the standard inclusion, i.e.
  $\iota(v,[z:w])=(v,[z:w])\in\setC^2\times\setC\setP^1=\underline{\setC^2}$
  and set
  \begin{eqnarray*}
    F_0:TS^2\rightarrow\underline{\setC^2},\qquad
    F_0=\iota\circ\tilde F_0
  \end{eqnarray*}
  $F_0$ is an injective complex vector bundle homomorphism
  because $\tilde F_0$ and $\iota$ are.
\end{proof}

\begin{figure}
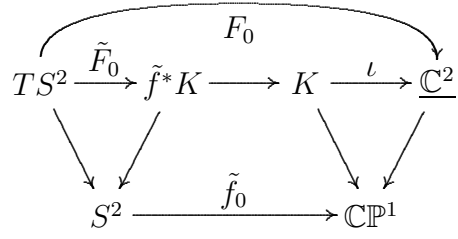
\label{fig:1}
  \centering
  \parbox{10cm}{
    \begindc{\commdiag}[5]
    \obj(5,20)[A]{$TS^2$}
    \obj(15,20)[B]{$\tilde f^\ast K$}
    \obj(25,20)[C]{$K$}
    \obj(35,20)[D]{$\underline{\setC^2}$}
    \obj(10,10)[a]{$S^2$}
    \obj(30,10)[b]{$\setC\setP^1$}
    \mor{A}{B}{$\tilde F_0$} \mor{B}{C}{} \mor{C}{D}{$\iota$}
    \mor{A}{a}{} \mor{B}{a}{} \mor{a}{b}{$\tilde f_0$} \mor{C}{b}{}
    \mor{D}{b}{}
    \cmor((5,22)(9,25)(20,26)(31,25)(35,22))
    \pdown(20,24){$F_0$}
    \enddc
  }
  \caption{The construction of the map $F_0$.}
\end{figure}

Set $f_0:S^2\rightarrow \setC^2$, $f_0(z)=0$ and let
$F_0:TS^2\rightarrow f_0^\ast T\setC^2=\underline{\setC^2}$ as in Lemma
\ref{lem:F_0}. We aim to construct an actual immersion
$f_1:S^2\rightarrow \setR^4$ so that $(f_1,F_1=df_1)$ has the same
topological data as $(f_0,F_0)$.

By Theorem 6.1 of \cite{hirsch} (or alternatively by the
$h$--principle) there exists an immersion $f_1:S^2\rightarrow \setR^4$
with $F_1=df_1:S^2\rightarrow f_1^\ast T\setR^4=\underline{\setR^4}$
together with a homotopy $f_t:S^2\rightarrow \setR^4$ connecting $f_0$
and $f_1$ covered by a homotopy of (real) monomorphisms
$F_t:TS^2\rightarrow \underline{\setR^4}$ connecting $F_0$ and
$F_1$. Here $\underline{\setR^4}$ is again the trivial bundle and we
implicitly made use of the canonical (real) isomorphism
$\underline{\setC^2}=\underline{\setR^4}$.

We need one more definition to construct a superregular operator with
non--trivial cokernel.  Let $\langle\cdot,\cdot\rangle$ denote the
standard inner product on $\setR^4$ and let $\J$ denote the space of
complex structures on $\setR^4$ and let $\A$ be the set of injective
(real--linear) homomorphisms from $\setC$ into $\setR^4$. We define
the map
\begin{eqnarray}\label{eq:Phi}
  \Phi:\A\rightarrow\J
\end{eqnarray}
in the following way. $A$ defines a splitting $\setR^4=V\oplus W$,
where $V=\image A$ and $W=V^\perp$. Define $J=\Phi(A)$ to be the
unique $J\in\J$ that leaves this splitting invariant, makes $A$
complex linear, and satisfies
\begin{eqnarray*}
  \langle Je_1,e_2\rangle=1,\qquad \langle Je_1,e_1\rangle=0
\end{eqnarray*}
on an oriented orthonormal basis $e_1,e_2$ of $W$.

\begin{lemma}\label{lem:D_super}
  There exists a superregular real--linear Cauchy--Riemann operator
  $D$ with non--trivial cokernel.
\end{lemma}

\begin{proof}
  For each $z\in S^2$ define $J(z)=\Phi\circ F_1(z)$, where $\Phi$ is
  the map from Equation (\ref{eq:Phi}). Note that $\Phi\circ
  F_0(z)=J_0$, so $\Phi\circ F_0$ is covered by the map
  $G_0=\id:S^2\rightarrow GL(4,\setR)$, where the projection
  $\pi:GL(4,\setR)\rightarrow \J$ is given by $\pi g=g^\ast
  J_0=g^{-1}\circ J\circ g$. The map  $\pi:GL(4,\setR)\rightarrow \J$ is a
  bundle projection and thus has the homotopy lifting property. Let
  $G_t$ be a lift of the homotopy $\Phi\circ F_t$ to $GL(4,\setR)$.

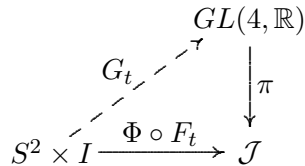
  \begin{center}\parbox{5cm}{
      \begindc{\commdiag}[5]
      \obj(5,5)[A]{$S^2\times I$}
      \obj(20,5)[J]{$\J$}
      \obj(20,15)[G]{$GL(4,\setR)$}
      \mor{A}{J}{$\Phi\circ F_t$}
      \mor{G}{J}{$\pi$}
      \mor(5,5)(18,15){$G_t$}[\atleft,\dasharrow]
      \enddc
    }\end{center}

  For $i=1,\ldots,4$ define sections $e_i(z)=G_1(z)\bar e_i$, where
  $\{\bar e_i\}_{i=1}^4$ denotes the canonical basis of $\setR^4$. By
  definition these are linearly independent for each $z\in S^2$, so by
  Lemma \ref{lem:Y} we can choose $Y$ so that these sections satisfy
  $D\,e_i=0$, where $D=D_0+Y$. Thus $D$ is superregular with
  superregular basis $\{e_i\}_{i=1}^4$.

  Set $e_5(z)=G_1(z)f_1(z)$, and note that $\delbar_{J(z)}f_1=0$ by
  the definition of $J(z)$. So by Lemma \ref{lem:J(z)-hol} $e_5$
  satisfies $D\,e_5=0$.

  Thus the kernel of $D$ is at least 5--dimensional, so $D$ has
  non--trivial cokernel.
\end{proof}

\begin{remark} \label{rmk:foln} Adding a suitable multiple of $e_4$ to
  $e_5$ we may assume that the pointwise inner product $\langle
  e_4(z),e_5(z)\rangle\ge 0$ but that the strict inequality does not
  hold. Then consider the map $$e:\setR^4 \times S^2 \to
  N,$$ $$(t_1,t_2,t_3,t_4,z)\mapsto \sum_{i=1}^3 t_i e_i + t_4 e_5 +
  t_4^2 e_4.$$ This is clearly a smooth map giving a foliation of $N$
  but its differential is not an isomorphism wherever $z$ satisfies
  $\langle e_4(z),e_5(z)\rangle = 0$.
\end{remark}

\subsection{Family of Almost Complex Structures}
\label{sec:constr-family-almost}

Here we will extend the example from the previous section to construct a family $D_s$ of Cauchy-Riemann operators for $s \in [-1,2]$ such that $D_s$ is regular for all $s$, superregular for $s$ close to $2$, but not superregular for $s$ close to $-1$. The example can be globalized to produce the counterexample needed for Theorem \ref{thm:counterexample}.

Let $N\rightarrow S^2$ be a trivial complex rank 2 vector bundle. Fix
an inner product $\langle\cdot,\cdot\rangle$ on N and let
\begin{eqnarray*}
  D:\Gamma N\rightarrow \Omega^{0,1}(N)
\end{eqnarray*}
be a superregular real--linear Cauchy--Riemann operator with
non--trivial cokernel as given by Lemma \ref{lem:D_super}. Let
$K=\ker(D)$ and fix a superregular basis $\{e_i\}_{i=1}^4$ and let
$e_5\in K$ be another section that is linearly independent from the
$\{e_i\}_{i=1}^4$.  Without loss of generality assume that $e_5$ is
perpendicular to $e_1,e_2,e_3$ in $L^2$ and the pointwise inner
product $h(z)=\langle e_4(z),e_5(z)\rangle\ge 0$ with
$h(0)=0$. Assume that $\norm{e_i}=1$, $i=1,\ldots,4$ and scale $e_5$
so that there exists $p_0\in S^2$ with
\begin{eqnarray}
  \label{eq:p_0}
  h(p_0)=\langle e_4(p_0),e_5(p_0)\rangle=1.
\end{eqnarray}

Let $C=\coker(D)$ be spanned by the orthonormal basis
$\{\eta_i\}_{i=1}^n\in C$. Let $\tilde D:K^\perp\rightarrow C^\perp$ be
the restriction of $D$ to $K^\perp$ and let $P:C^\perp\rightarrow
K^\perp$ denote its inverse. Let
\begin{eqnarray*}
  \pi_C:\Omega^{0,1}(N)\rightarrow C
\end{eqnarray*}
denote the orthogonal projection in $L^2$.

\begin{lemma}\label{lem:L-Y}
  There exists family of vector bundle homomorphisms $Y_s:N\rightarrow
  \Lambda^{0,1}(T^\ast S^2\otimes_\setC N)$, $s\in[-1,1]$ so that
  \begin{eqnarray*}
    L_s=\pi_C\circ Y_s:\Gamma(N)\rightarrow C
  \end{eqnarray*}
  is surjective and
  \begin{eqnarray*}
    K_s=\ker(L_s)\cap K=span\{e_1,e_2,e_3,s\,e_4+(1-s)e_5\}.
  \end{eqnarray*}
\end{lemma}

In order to prove this we need the following three lemmas. We will make
use of the notation introduced above. Also recall the point $p_0\in
S^2$ from Equation (\ref{eq:p_0}), so the family of sections $e_1^s$
defined via $e_i^s=e_i$ for $i=1,\ldots, 3$ and
$e_4^s=s\,e_4+(1-s)e_5$ are linearly independent vectors over $p_0$
for all $s\in[-1,1]$. Let $U\subset S^2$ be an open neighborhood of
$p_0$ so that the $\{e^s_i(z)\}_{i=1}^4$ still are linearly
independent vectors over all $z\in \overline U$. Let $V\subset
C^\perp$ denote the subspace of smooth sections in $C^\perp$ that are
supported in $U$. Then we can define a family of homomorphisms $Y_s$
from $N$ to $\Lambda^{0,1}(T^\ast S^2\otimes_\setC N)$ via functions
$g_i^s:[-1,1]\times U\rightarrow V$ by
\begin{eqnarray*}
  Y_s(e_i^s)=g_i^s.
\end{eqnarray*}
Set $L_s=\pi_C\circ Y_s:\Gamma(N)\rightarrow C$ and
$K_s=\ker(L_s)$. Note that by construction $e^s_i\in K_s$ for all
$i=1,\ldots 4$. We will prove Lemma \ref{lem:L-Y} by finding suitable
$g_i^s$ so that $L_s$ is surjective. Alternatively we can define a map
\begin{eqnarray}\label{eq:F}
  F:[-1,1]\times V^4\rightarrow [-1,1]\times Hom(K,C),\qquad
  F(g_1^s,\ldots,g_4^s)=L_s|_{K}
\end{eqnarray}
and we need to show that the image of $F$ contains a family
of surjective homomorphisms.

\begin{lemma}\label{lem:1}
  Suppose that $g_i^s$ are given so that $\mathrm{dim}(K_s) \le p+4$
  for some $0\le p\le m$ and all $s\in[-1,1]$.  Then there exists a
  smooth family of linearly independent sections $\{f_j^s\}_{j=1}^m\in
  K$ that are orthogonal to $e_i^s$ for $i=1,\ldots,4$ such that
  $\{f_j^s\}_{j=p+1}^m \in K_s^\perp \cap K$ for all $s$, and
  $<L_s(f_j),L_s(f_k)>=0$ for all $j \le p$ and $k>p$.
\end{lemma}

\begin{proof}
  Define $S_{p+4} \subset S_{p+3} \subset ... \subset S_4 = [-1,1]$
  where $S_r = \{s \in [-1,1] | \mathrm{dim}(K_s) \ge r\}$. Restricted
  to $S_{p+4}$ the vector spaces $K_s^{\perp} \cap K$ form a smooth
  vector bundle and so admit a smooth frame of $m-p$ sections
  $\{f_j^s\}_{j=p+1}^m$. But suppose that we have defined an
  $m-p$--frame (that is, $m-p$ linearly independent sections) over
  $S_r$. The vector spaces $K_s^{\perp} \cap K$ also form a continuous
  vector bundle over $S_{r-1} \setminus S_r$ which extend to a
  continuous vector bundle over $S_r$. Thus by the Tietze extension
  theorem the sections $\{f_j^s\}_{j=p+1}^m$ extend to $S_{r-1}$ and
  we can conclude by induction to define these sections over
  $[-1,1]$. The other sections $f_j^s$ can then be constructed
  smoothly over $s$ since the orthogonal complement of the sections
  already constructed forms a vector bundle over $[-1,1]$ which
  therefore admits a smooth frame field. We can arrange that their
  images are othogonal by a Gram-Schmidt procedure.
\end{proof}

\begin{lemma}\label{lem:3}
  Let $\hat K\subset K$ and $\hat C\subset C$ be $p$--dimensional
  vector spaces. Then the map 
  \begin{eqnarray*}
    \hat F:V^4\rightarrow Hom(\hat K,\hat C),\qquad
    \hat F(g_1,\ldots g_4)=\pi_{\hat C}\circ Y|_{\hat K},
  \end{eqnarray*}
  where $Y:N\rightarrow \Lambda^{0,1}(T^\ast S\otimes_\setC N)$ is
  the homomorphism associated to the $g_i$, is nonzero.
\end{lemma}

\begin{proof}
  Let $a_{ij}$ be such that $f_j^s=\sum_{i=1}^4 a_{ij}^se_{i}^s$ on
  $U$. Then $\hat F(g_1,\ldots g_4)(f_j^s)=\sum_{i=1}^4 g_i^s
  a_{ij}^s$ on $U$.

  If $\hat F$ was trivial, then the $L^2$ inner product
  \begin{eqnarray*}
    \sum_{i=1}^4\langle g_i^s\cdot a_{ij}^s,\eta_k\rangle=0\quad
    \forall\,k,j\quad\forall g_i^s\in V
  \end{eqnarray*}
  Thus $(a_{ij}^s\eta_k)_{i=1}^4\in (V^4)^\perp$ for all $j,k$.  Recall
  that $K$ and $C$ are the kernel and cokernel of a real--linear
  Cauchy--Riemann operator $D$, respectively. Then for fixed $j,k$, if
  $a_{ij}^s\eta_k\in V^\perp$ for all $i$, we have over $U$
  \begin{eqnarray*}
    D^\ast (a_{ij}^s\eta_k)=(\delbar^\ast a_{ij}^s)\otimes
    \eta_k=0\quad
    \forall\,i.
  \end{eqnarray*}
  On the other hand 
  \begin{eqnarray*}
    0=D\,f_j=\sum_{i=1}^4 D(a_{ij}^se_i^s)=\sum_{i=1}^4 (\delbar
    a_{ij}^s)\otimes e_i^s
  \end{eqnarray*}
  and the two equations taken together imply that the
  $\{a_{ij}^s\}_{i=1}^4$ are constant functions over $U$. This means
  that $f_j^s$ is a linear combination of the $e_i^s$ on $U$. By
  unique continuation, using that $f_i^s$ and $e_i$ are in the kernel
  of the operator $D$, we conclude that $f_j^s$ is globally a linear
  combination of the $e_i$.  But that contradicts that $f_j^s$ is
  linearly independent of the $e_i^s$.
\end{proof}

\begin{lemma}\label{lem:2}
  Suppose that $Y_s$ and corresponding sections $f_j$ are given as in
  Lemma \ref{lem:1}. Then there exist a smooth family $(g^s_i)_{i=1}^4
  \in V^4$ such that for each of the corresponding maps $L_s$ there
  exists an $f^s \in \mathrm{span}\{f_j^s\}_{j=1}^p$ with $L_s(f^s)\ne
  0$ and independent of the span of the $L_s(f_j)$ for all $j>p$.
\end{lemma}

\begin{proof}
  As in the previous lemma we express the $\{f_j^s\}_{j=1}^p$ as a
  linear combination of functional multiples of the
  $\{e_i^s\}_{i=1}^4$, at least over the open set $U$.  Then we
  redefine the above linear map $F$ as in Equation (\ref{eq:F}) such
  that its range is paths of $p \times p$ matrices with entries
  determined by the $L^2$ inner products of the $\{f_j^s\}_{j=1}^p$
  with $\{\eta_j^s\}_{j=1}^p$ orthogonal to the $L_s(f_j)$ for all
  $j>p$.  This linear map is nonzero for all $s$ by Lemma
  \ref{lem:3}. Therefore it has a positive codimensional kernel $M^s
  \subset V^4$ for all $s$ and the lemma follows if we can find a
  continuous section of $M^{\perp}$ over $[0,1]$. But the rank of
  these vector spaces is again lower semicontinuous in $s$ and so does
  indeed admit a section as above by first defining over points of
  minimal rank and then extending as before in the proof of Lemma
  \ref{lem:1}.
\end{proof}

\begin{proof}[Proof of Lemma \ref{lem:L-Y}]
  We prove this by induction by perturbing $g^s_i$.  Suppose that we
  have found $g_i^s$ such that $\mathrm{dim}\ker(K_s) \le p+4$ for all
  $s$ and some $0<p\le m$. Then we can apply Lemma \ref{lem:1} to find
  corresponding families of sections and thus a perturbation of the
  $g_i^s$ using Lemma \ref{lem:2}. If the $g_i^s$ are chosen
  sufficiently small then the $L_s(f_j)$ are still linearly
  independent for all $s$ and $j=p+1,..,m$. However for each $s$ there
  is now an $f^s \in \mathrm{span}(\{f_j^s\}_{j=1}^p)$ with
  $L_s(f^s)\ne 0$ and independent to the $L_s(f_j)$ for
  $j>p$. Hence for each $s$ we have that $\mathrm{dim}(K_s) \le p+4-1$
  and the proof follows.
\end{proof}

Let $Y_s$ and $L_s$ be as in Lemma \ref{lem:L-Y} and consider the family
of real--linear Cauchy--Riemann operators
\begin{eqnarray}\label{eq:D_st}
  D_{s,t}=D+t\,Y_s.
\end{eqnarray}

The kernel of $D_{s,t}$ gets arbitrarily close to $K_s$ as $t$ gets
small as described below. This result is well established in the
literature (see e.g. \cite{kato}), but we give a proof here for the
convenience of the reader.  In the following $\norm{\cdot}$ denotes
the $L^2$--norm.
\begin{lemma}\label{lem:D_st}
  There exists a constant $c>0$ so that for all $0<|t|<1/2c$, $s\in
  [-1,1]$ and $v_s\in K_s$, $D_{s,t}$ is surjective and there exists a
  unique $\xi_{s,t}(v_s)\in K_s^\perp$ so that
  \begin{eqnarray*}
    v_s+\xi_{s,t}(v_s)\in\ker D_{s,t}.
  \end{eqnarray*}
  Moreover $\norm{\xi_{s,t}(v_s)}\le 2tC\norm{v_s}$.
\end{lemma}
\begin{proof}
  Let $V_s=K_s^\perp\cap K$ denote the orthogonal complement of
  $K_s$ in $K$ and set
  \begin{eqnarray*}
    \tilde L_s=L_s|_{V_s}:V_s\rightarrow C.
  \end{eqnarray*}
  Then $\tilde L_s$ is an isomorphism.

  Let $W_s=K_s^\perp\cap \ker L_s\subset L^2(N)$ and consider the
  compact operator
  \begin{eqnarray*}
    F_s:\ker L_s\rightarrow W_s,\qquad
    F_s(\zeta)=\tilde L_s^{-1}\circ L_s(PY_s(\zeta))-PY_s(\zeta),
  \end{eqnarray*}
  where $P:C^\perp\rightarrow K^\perp$ is the inverse of
  $D|_{K^\perp}:K^\perp\rightarrow C^\perp$. 

  Note that $L_s\circ F_s=0$ and $\tilde L^{-1}$ and $P$ have
  image in $K_s^\perp$, so $F_s$ is well defined. Let
  $c=\sup_{s\in[-1,1]}\norm{F_s}_{L^2}$. For $|t|c\le \frac12$ and
  $v\in \ker L_s$ note that
  \begin{eqnarray*}
    \norm{\sum_{n=1}^N t^n F_s^n(v)}
    \le\sum_{n=1}^N |t|^n \norm{F_s^n(v)}
    \le\sum_{n=1}^N |t|^nc^n\norm{v}
    < 2|t|c\norm{v}
  \end{eqnarray*}
  so we may define
  \begin{eqnarray*}
    \xi_{s,t}(v)=\sum_{n=1}^\infty t^n F_s^n(v),
  \end{eqnarray*}
  also satisfying $\norm{\xi_{s,t}(v)}\le 2tc\norm{v}$. Moreover,
  \begin{eqnarray*}
    D_{s,t}\circ F_s=D\circ F_s+tY_s\circ F_s=-Y_s+tY_s\circ F_s
  \end{eqnarray*}
  and thus
  \begin{eqnarray*}
    D_{s,t}\left(v+\sum_{n=1}^N t^nF_s^n(v)\right)
    &=&Dv+tY_s(v)+\sum_{n=1}^N\left(
      t^{n+1}Y_sF_s^n(v)-t^nY_sF_s^{n-1}(v)\right)\\
    &=&Dv+t^{N+1}Y_sF_s^N(v),
  \end{eqnarray*}
  which converges strongly to $Dv$ (in $L^2$) as $N\rightarrow\infty$, so
  \begin{eqnarray*}
    D_{s,t}(v+\xi_{s,t}(v))=Dv,\qquad \forall v\in \ker L_s.
  \end{eqnarray*}
  In particular $v+\xi_{s,t}(v)\in\ker D_{s,t}$ for all $v\in K_s$.

  Next we show that $D_{s,t}$ is surjective. It suffices to show that
  the image of $D_{s,t}$ is dense as $D_{s,t}$ is Fredholm and thus
  has a closed image. By Hahn--Banach, it suffices to show that there
  does not exists $0\ne\mu\in L^2(\Lambda^{0,1}(T^\ast
  S^2\otimes_\setC N))$ that annihilates
  the image of $D_{s,t}$. Suppose to the contrary such a $\mu$
  exists. Write $\mu=\mu_0+\mu_1$, where $\mu_0\in C$ and $\mu_1\in
  C^\perp$. Without loss of generality assume that $\mu_1\ne0$,
  otherwise, for $\zeta=\tilde L_s^{-1}(\mu_0)\in V_s$,
  \begin{eqnarray*}
    \langle D_{s,t}\zeta,\mu\rangle
    =\langle tY_s(\tilde L_s^{-1}(\mu_0)),\mu_0\rangle
    =t\langle \tilde L_s(\tilde L_s^{-1}(\mu_0)),\mu_0\rangle
    =t\norm{\mu_0}^2=t\norm{\mu}^2\ne 0.
  \end{eqnarray*}

  Set $\zeta=P(\mu_1)-\tilde L_s^{-1}\circ L_s\circ P(\mu_1)\in \ker
  L_s$ and consider $\zeta+\xi_{s,t}(\zeta)$.  Then
  \begin{eqnarray*}
    \langle D_{s,t}(\zeta+\xi_{s,t}(\zeta)),\mu\rangle
    =\langle D\zeta,\mu\rangle
    =\langle \mu_1,\mu\rangle=\norm{\mu_1}^2\ne 0.
  \end{eqnarray*}
  This shows that $D_{s,t}$ is surjective.

  The uniqueness of $\xi_{s,t}(v)$ satisfying $D_{s,t}(v+\xi_{s,t}(v))$
  follows from the surjectivity of $D_{s,t}$.
\end{proof}

In particular the above Lemma guarantees that for any given $\delta>0$
there exists $t_0>0$ so that for all $t<t_0$ the regular operators
$D_{s,t}$ are superregular for $s\in[\delta,1]$ and $D_{s,t}$ are not
superregular for $s\in[-1,-\delta]$. To see this note that for $s$ in
that range the quantity $e_4^s=s\,e_4+(1-s)e_5$ satisfies $\langle
e_4(0),e_4^s(0)\rangle<0$ and $\langle e_4(p_0),e_4^s(p_0)\rangle=1$
by Equation (\ref{eq:p_0}).  Thus near $p_0$ the tuple
$(e_1,e_2,e_3,e_4^s)$ forms an oriented basis of $\setR^4$ and at the
point $0$ they form a basis with the opposite orientation. In
particular there must be points in $S^2$ where the sections do not
form a basis of $\setR^4$. This remains true under small perturbations
of the tuple $(e_1,e_2,e_3,e_4^s)$.

A real--linear Cauchy Riemann operator $D$ on $N$ gives rise to an
$\setR$--invariant almost complex structure $J$ on the total space of
$N$ in the following way. Choose a local complex trivialization
$N=S^2\times \setC^2$ and write $D=\delbar_0+\frac12 Y\circ j$, where
$Y\in \mathrm{Hom}_\setR(\setC^2,\Lambda^{0,1}T^\ast
S^2\otimes\setC^2)$. Utilizing the projections to each factor $S^2$
and $\setC^2$ of $N$, referred to as the horizontal and vertical
directions with complex structures $j$ and $i$, respectively, we
define the almost complex structure $J$ at a point $x=(w,u)\in N$
acting on a vector $(h,v)\in T_xN$ via
\begin{eqnarray*}
  J(h,v)=jh+Y_{(w,u)}h+iv.
\end{eqnarray*}
Note that $J$ is independent of the trivialization chosen and indeed
satisfies $J^2=-\id$. Moreover, if $f:S^2\rightarrow N$ with
$f(z)=(w(z),u(z))\in S^2\times \setC^2$ in the homology class of a
section, then
\begin{eqnarray*}
  \delbar_J f=\frac12\left\{df+J\,df\circ j_0\right\}
  =\frac12\left\{dw+j\,dw\circ j_0\right\}
  +\frac12\left\{du+i\,du\circ j_0+Y_{(w,u)} dw\circ j_0\right\}
\end{eqnarray*}
Thus for $\delbar_Jf=0$ it is necessary that $w(z)=z$ and $j=j_0$, up
to a diffeomorphism of the domain $S^2$. In that case
\begin{eqnarray*}
   \delbar_Jf=\frac12\left\{du+i\,du\circ j_0+Y_{(w,u)}\circ j\right\}
   =D(u)
\end{eqnarray*}
so maps $f:S^2\rightarrow N$ in the class of a section are
$J$--holomorphic if and only if they can be parametrized as a section
$f(z)=(z,\xi(z))$ and $D\xi=0$. Moreover note the the zero section is
always a $J$--holomorphic section no matter what $D$ is and that the
linearization of $\delbar_J$ at the zero section is $D$.

Let $\omega$ be the canonical product symplectic form on $N$ so that
on each fiber it reduces to the Fubini-Study form and let $\tilde J$
be the canonical product complex structure on $N$ and $\tilde D$ the
associated Cauchy--Riemann operator. Given any symplectic $4$-manifold
$(M,\omega_M)$ there exists a symplectic embedding from $U$ into
$(X,\omega)=(S^2 \times M, \sigma_0 \oplus \omega_M)$ preserving the
$S^2$ factors, where $U$ is a suitable small neighborhood of the
zero-section in $N$. Thus $\tilde J$ extends to a product complex
structure on $X$ which is tamed by $\omega$, and $X$ is smoothly
foliated by regular $\tilde J$-holomorphic spheres.

Let $D_s$, $s\in[-1,2]$ be a smooth family of real--linear Cauchy
Riemann operators on $N$ so that $D_s=D_{s,t}$ for some small fixed
$t$ and $s\in[-1,1]$, where $D_{s,t}$ is the operator from Lemma
\ref{lem:D_st}, and $D_s$ interpolates between $D_1$ and $\tilde D$
for $s\in[1,2]$. Denote the associated family of almost complex
structures by $J_s$. Note that $J_s$ are tamed by $\omega$ on a
neighborhood $U$ of the zero section in $N$. We now modify the family
$J_s$ to construct a family of almost complex structures $\tilde J_s$
on $N$ with the property that $\tilde J_s=\tilde J$ outside of $U$ and
$\tilde J_s=J_s$ in an open neighborhood $V\subset U$ of the zero
section so that $\tilde J_s$ is tamed by $\omega$. Using the above
embedding we similarly construct the family $\tilde J_s$ on $X$.

The family $\tilde J_s$ is tamed by the canonical symplectic structure
on $X=S^2\times M$, and $\tilde J_2$ is the product complex
structure on $X$. Thus $\tilde J_2$ is regular (and superregular) and
$X$ is foliated by $\tilde J_2$--holomorphic spheres. By construction
$\tilde J_s$ is regular for all curves outside of $U$ and inside of
$V$. By possibly adding a small perturbation to the family $\tilde
J_s$ over $U\setminus V\subset X$ we may assume that the family
$\tilde J_s$ is a regular family of almost complex structure and that
$\tilde J_{-1}$ is regular.

Since $\tilde J_s=\tilde J$ outside of $U$, the complement of $U$ is
foliated by $\tilde J$--holomorphic spheres for all
$s\in[-1,2]$. Moreover, the zero section is $\tilde J_s$--holomorphic
for all $s\in[-1,2]$. But the linearized operator at the zero section
is $D_s$, which is not superregular for $s=-1$ by construction, so the
foliation does not persist to a $J_{-1}$--holomorphic foliation of
$X$.

This proves Theorem \ref{thm:counterexample} in the case that
$(X,\omega)=(S^2\times M^4,\sigma_0\times \omega_M)$.

\section{Stability of Foliations for Integrable Complex Structures}
\label{sec:stab-foli-integr}

In this section we show that holomorphic foliations are stable under
perturbations of complex structure so that the holomorphic bisectional
curvature (see below or e.g. \cite{kobayashi_nomizu_2}) remains
bounded.
\begin{definition}
  Let $(M,J)$ be a complex manifold. The {\em holomorphic bisectional
    curvature} $H(p,p')$ of two $J$--invariant planes $p$
  and $p'$ in $T_xM$ is
  \begin{eqnarray*}
    H(p,p')=R(X,JX,Y,JY)
  \end{eqnarray*}
  where $R$ is the regular Riemannian curvature tensor and $X$ and $Y$
  are unit vectors in $p$ and $p'$, respectively.

  We say that $(M,J)$ has holomorphic bisectional curvature bounded
  from above (below) by a constant $c$ if $H(p,p')\le c$ ($H(p,p')\ge
  c$) for all $x\in M$ and $J$--invariant planes $p,p'\subset T_xM$.
\end{definition}

The following result is a computation from page 79 of
\cite{griffiths_harris}.
\begin{lemma}\label{lem:quotient_curvature}
  Let $G\rightarrow M$ be a holomorphic vector bundle of (complex)
  rank at least 2 over a complex
  manifold $M$, and let $E\subset G$ be a holomorphic subbundle and
  $F=E^\perp$ the orthogonal complement of $E$ in $G$. Then the local
  curvature form of $F$ is greater than or equal to the curvature
  form of $G$ restricted to $F$.
\end{lemma}

\begin{lemma}\label{lem:cur_k}
  Let $(X,\omega,J)$ be K\"ahler and
  let $u:S^2\rightarrow X$ be a $J$--holomorphic
  sphere. Assume that the the holomorphic bisectional curvature of
  $(X,J)$ is bounded from below by $c>\pi k/\omega[u]$. Then the
  pullback bundle $u^\ast TX$ has no holomorphic line-subbundle with
  first Chern class less than or equal to $k$.
\end{lemma}
\begin{proof}
  Let $F$ be a holomorphic line-subbundle of $u^\ast TX$, and let $E$
  be a complementary holomorphic subbundle, which exists by a result
  of Grotherndieck \cite{grothendieck} if the real dimension of $X$ is at
  least 4 and is taken to be empty otherwise. Let $E^\perp$ denote the
  orthogonal complement
  of $E$ in $u^\ast TX$ and denote the curvature of $E^\perp$ with
  respect to the connection induced by $u^\ast TX$ by $K$. By Lemma
  \ref{lem:quotient_curvature} we know that $K(\cdot)\ge u^\ast
  H(\cdot,F)$ on any complex frame. Then
  \begin{eqnarray*}
    c_1(F)=c_1(E^\perp)=\frac1{2\pi}\int_{S^2}K
    \ge\frac1{2\pi}\int_{S^2}u^\ast H(\cdot,F)
    \ge\frac1{2\pi}c\int_{S^2}\norm{du}^2\dvol
    >k.
  \end{eqnarray*}
\end{proof}

Recall Definition \ref{def:superregular} for our use of the terms regular and superregular.
\begin{lemma}\label{lem:superregular}
  Let $(X,\omega,J)$ be K\"ahler so that the holomorphic bisectional
  curvature is bounded from below by $c>-2\pi/\omega(A)$, where $A\in
  H_2(X;\setZ)$.  Then any $J$--holomorphic sphere in the class of $A$
  is regular.

  If furthermore $u$ is immersed, $c_1(A)=2$, and the holomorphic
  bisectional curvature is bounded from below by $c>-\pi/\omega(A)$,
  then $u$ is also superregular.
\end{lemma}

\begin{proof}
  Let $u:S^2\rightarrow X$ be a $J$--holomorphic curve representing
  $A$. By Lemma \ref{lem:cur_k}, and using that
  $c>-2\pi/\omega(A)$, the pullback tangent bundle $u^\ast TX$ does
  not have a holomorphic
  line-subbundle with first Chern class less than $-1$. Thus $u$ is
  regular by Lemma 3.3.1 in \cite{mcduff2}.

  If furthermore $u$ is immersed and $c>-\pi/\omega(A)$, then every
  holomorphic line-subbundle of the normal bundle has first Chern
  class $\ge 0$. Since $u$ is immersed and $c_1(A)=2$, the first Chern
  class of the normal bundle is 0. Thus any holomorphic line-subbundle
  of the holomorphic normal bundle has first Chern class 0 and has a
  superregular basis. So the normal bundle to $u$ has a superregular
  basis and $u$ is superregular.
\end{proof}

\begin{corollary}\label{cor:integrable_foln}
  Let $J_I^c$ be the space of integrable compatible complex structures
  on $(X^{2n},\omega)$ with holomorphic bisectional curvature bounded from
  below by $c>-2\pi/\omega(A)$, where $A\in H_2(X;\setZ)$ with
  $c_1(A)=2$. Further assume that there exists a $J_0$--holomorphic
  foliation of $X$ by spheres in the class of $A$ for some
  $J_0\in\J_I^c$.

  Then $X$ is foliated by $J$--holomorphic spheres for any
  $J\in\J_I^c$ in the path--component of $J_0$.
\end{corollary}

\begin{proof}
  Let $J=J_1\in\J_I^c$ be connected to $J_0$ via a path $\{J_t\}_{t\in
    [0,1]}$ and let $\M$ denote the family space of
  $\{J_t\}_{t\in [0,1]}$--holomorphic spheres in the class of
  $A$ as in Equation (\ref{eq:M}) and let
  \begin{eqnarray*}
    \M^1=\M\times_G S^2
  \end{eqnarray*}
  denote the component of the one--pointed moduli
  space, modulo automorphisms. 

  By Lemma \ref{lem:superregular} all curves in $\M^1$ are regular,
  so $\M^1$ is a smooth manifold (of dimension $2n+1$) and the
  projection onto the $[0,1]$--factor is a submersion.

  Denote the connected component of $\M^1$ containing the initial
  $J_0$--holomorphic foliation by $\tilde M$ and let $\tilde
  M_s=\{(u,p,t)\in \tilde M|\,t=s\}\subset \M$. Again, by Lemma
  \ref{lem:superregular}, all curves in $\M_s$ are regular, so $\M_s$
  is a smooth manifold (of dimension $2n$). The evaluation map
  $ev_s:\M_s\rightarrow X$ is holomorphic with respect to the natural
  complex structure on $\M_s$. It has degree 1, since $ev_0$ is of
  degree 1 by assumption. Thus $ev_s$ is a diffeomorphism for all
  $s\in[0,1]$ and $\M_s$ is a smooth foliation of $X$. 
\end{proof}

\begin{remark}
  Note that any $J$--holomorphic sphere $u$ (for integrable $J$) that
  is part of a smooth foliation is automatically regular and
  superregular. Indeed, any line subbundle of the normal bundle of $u$
  has non-negative first Chern class, since it has holomorphic
  sections induced by nearby curves. Since the first Chern class of
  the normal bundle at a leaf of a foliation is trivial all linear
  subbundles must have first Chern class 0, so the curve is regular
  and superregular.
\end{remark}

Under more stringent curvature assumptions, we can prove the existence
of foliations given conditions on a Gromov-Witten invariant. This is
the substance of Theorem \ref{thm:integrable} that we are now prepared
to prove.

\begin{proof}[Proof of Theorem \ref{thm:integrable}]
  Let $J\in\J_I^c$ and let $\M$ be the space of $J$--holomorphic
  spheres in the class $A$. $\M$ is non-empty since the GW count is
  non-zero. Let $u\in\M$. By Lemma \ref{lem:cur_k} we know that all
  holomorphic linear subbundles of $u^\ast TX$ have first Chern class
  greater than or equal to 0 and $u$ is regular. We claim that $u$ is
  immersed. If not, then the first Chern class of the tangent bundle
  is at least 4. But by the dimension formula we know that $c_1(A)=2$,
  so the holomorphic normal bundle would contain a linear subbundle
  with first Chern class less than 0 which is impossible. Thus $u$ is
  superregular.

  Since every $u\in\M$ is regular and superregular, the evaluation map
  is transverse to any $x\in X$ and the curves contribute positively to
  the GW count of a point class. Thus the evaluation map
  $ev:\M\rightarrow X$ has degree one and is holomorphic, so it is a
  diffeomorphism, showing that $X$ is foliated by embedded holomorphic
  spheres.
\end{proof}

\providecommand{\bysame}{\leavevmode\hbox to3em{\hrulefill}\thinspace}
\providecommand{\MR}{\relax\ifhmode\unskip\space\fi MR }
\providecommand{\MRhref}[2]{%
  \href{http://www.ams.org/mathscinet-getitem?mr=#1}{#2}
}
\providecommand{\href}[2]{#2}

\end{document}